\documentclass[12pt]{article}
\usepackage{color, amsmath,amssymb, amsfonts, amstext,amsthm, latexsym}
\usepackage[active]{srcltx}
\usepackage{extarrows}
\usepackage{amsfonts,amsmath,amsthm, amssymb}
\usepackage{latexsym, euscript, epic, eepic}
\usepackage{lineno}

\newtheorem{defn}{Definition}[section]

\newtheorem{thm}{Theorem}[section]
\newtheorem{pro}{Proposition}[section]
\newtheorem{exmp}{Example}[section]
\newtheorem{rmk}{Remark}[section]

\begin{document}

\title{ {\bf Topological pressure dimension for almost additive potentials}
 \footnotetext {* Corresponding author}
  }
\author{Lei Liu$^{1,2}$, Huajun Gong$^{2}$, Xiaoyao Zhou$^{2*}$,  \\
\small 1 School of Mathematics and Information Science, Shangqiu Normal University,\\
    \small Shangqiu  476000, Henan P.R.China,\\
     \small  2 Department of Mathematics, University of Science and Technology of China,\\
      \small  Hefei, 230026, Anhui, P.R.China\\
       \small    e-mail: mathliulei@163.com\\
        \small huajun84@hotmail.com\\
         \small    zhouxiaoyaodeyouxian@126.com,\\
}
\date{}
\maketitle

\begin{center}
 \begin{minipage}{155mm}
{\small {\bf Abstract.} This paper is devoted to the study of the topological pressure
dimension for almost additive sequences, which is an extension of topological
entropy dimension. We investigate fundamental properties of the topological pressure dimension
for almost additive sequences. 
In particular, we study the relationships among different types of topological pressure dimension
and identifies an inequality relating them. Also, we show that the topological pressure dimension
is always equal to or greater than 1 for certain special almost additive sequence.}
\end{minipage}
 \end{center}

\vskip0.5cm

%%%%%%%%%%%%%%%%%%%%%%%%%%%%%%%%%%%%%%%%%%%%%%%%%%%%%%%%%%%%%%%%%%%%%%%%%%%%%%%%
%%%%%%%%%%%%%%%%%%%%%%%%%%%%%%%%%%%%%%%%%%%%%%%%%%%%%%%%%%%%%%%%%%%%%%%%%%%%%%%%
%%%%%%%%%%%%%%%%%%%%%%%%%%%%%%%%%%%%%%%%%%%%%%%%%%%%%%%%%%%%%%%%%%%%%%%%%%%%%%%%
%%%%%%%%%%%%%%%%%%%%%%%%%%%%%%%%%%%%%%%%%%%%%%%%%%%%%%%%%%%%%%%%%%%%%%%%%%%%%%%%
\section{Introduction}
\qquad In 1958, Kolmogorov applied the notion of entropy from
information theory to ergodic theory. Since then,
the concepts of entropies are useful for studying topological
and measure-theoretic structures of dynamical systems, that is,
topological entropy (see \cite{Adler-Konheim-McAndrew, Bowen1,
Bowen2}) and measure-theoretic entropy (see
\cite{Cornfeld-Fomin-Sinai-book, Kolmogorov-Tihomiorov}). For
instance, two conjugate systems have the same entropy and thus entropy
is a numerical invariant of the class of conjugated dynamical
systems. The theory of expansive dynamical systems has been closely
related to the theory of topological entropy \cite{Bowen-Walters,
Keynes-Sears, Thomas}. Entropy and chaos are closely related, for
example, a continuous map of interval is chaotic if and only if it
has a positive topological entropy \cite{Block-Coppel-book}. Moreover,
various authors have introduced several refinements of entropy,
including slow entropy \cite{Katok-Thouvenot}, measure-theoretic
complexity \cite{Ferenczi}, and entropy-like invariants
for noninvertible maps\cite{Hurley,Nitecki-Przytycki,Cheng-Newhouse}.
The authors \cite{Cheng-Li2,MaKuangLi,MaWu,KuangChengLi}
extended and studied as above some entropy-like
invariants for the non-autonomous discrete dynamical systems given
by a sequence of continuous self-maps of a compact topological space.

Properties of positive entropy systems have been studied in many different
respects along with their applications. Comparing with positive
entropy systems, we have much less understanding and less tools for
zero entropy systems, although systems with positive entropy are much
more complicated than those with zero entropy. Zero entropy systems have
various complexity, and have been studied many authors
(see \cite{Carvalho,Cheng-Li,Dou-Huang-Park1,Dou-Huang-Park2,
Ferenczi-Park,Huang-Yi,Huang-Park-Ye,Misiurewicz-Smital,Park}). These authors
adopted various methods to classify zero entropy dynamical systems.
Carvalho \cite{Carvalho} introduced the notion of entropy dimension
to distinguish the zero topological entropy systems and obtained
some basic properties of entropy dimension. Cheng and Li \cite{Cheng-Li} further
discuss entropy dimension of the probabilistic and the topological
versions and gives a symbolic subspace to achieve zero topological
entropy, but with full entropy dimension.
Ferenczi and Park \cite{Ferenczi-Park} investigated a new entropy-like
invariant for the action of $\mathbb{Z}$ or $\mathbb{Z}^d$ on a
probability space.

Topological pressure is a generalization of topological entropy for a dynamical
system. The notion was first introduced by Ruelle \cite{Ruelle} in 1973 for
expansive maps acting on compact metric spaces. In the same paper he formulated
a variational principle for the topological pressure. Later Walters \cite{Walters1}
generalized these results to general continuous maps on compact metric spaces.
The theory concerning the topological pressure, variational principle and equilibrium
states plays a fundamental role in statistical mechanics, ergodic theory and
dynamical systems (see \cite{Bowen3,Keller,Ruelle1,Walters2}). After the works
of Bowen \cite{Bowen4} and Ruelle \cite{Ruelle2}, the topological pressure
turned into a basic tool in the dimension theory related to dynamical systems.
In 1984, Pesin and Pitskel \cite{Pesin-Pitskel} defined the topological pressure of additive
potentials for non-compact subsets of compact metric spaces and proved the variational
principle under some supplementary conditions. In 1988, Falconer \cite{Falcone}
considered the thermodynamic formalism for sub-additive
potentials for mixing repellers. He proved the variational principle for the topological pressure
under some Lipschitz conditions and bounded distortion assumptions on the sub-additive potentials.
In 1996, Barreira \cite{Barreira1} extended the work of Pesin and Pitskel. He defined the topological
pressure for an arbitrary sequence of continuous functions on an arbitrary subset of compact
metric spaces, and proved the variational principle under a strong convergence assumption on
the potentials. In 2008, Cao, Feng and Huang \cite{Cao-Feng-Huang} and Feng and Huang \cite{Feng-Huang} generalized
Ruelle and Walters¡¯s results to sub-additive potentials in general compact dynamical systems.
Zhang \cite{Zhang}
introduced the notion of measure-theoretic pressure for sub-additive potentials, and studied the
relationship between topological pressure and measure-theoretic pressure. Recently, Chen,
Ding and Cao \cite{Chen-Ding-Cao} studied the local variational principle of
topological pressure for sub-additive potential, Liang and Yan \cite{Liang-Yan} introduced
the topological pressure for any sub-additive potentials of a countable discrete amenable
group action and established a local variational principle for it, and Yan \cite{Yan}
investigated the topological pressure for any sub-additive and asymptotically
sub-additive potentials of $\mathbb{Z}^d$-actions and established the variational principle
for them. Cheng and Li \cite{Cheng-Li2} extended the definition of entropy dimension
and gave the definition of topological pressure dimension, which is also similar to
the fractal measure, and studied the relationships among different types of topological
pressure dimension and identifies an inequality relating them.

In this paper we introduce different various topological
pressure dimension for almost additive sequences and discuss how they are related
to one another. The defined topological pressure dimension generates that of Cheng and Li.
We study properties of the different forms of
topological pressure dimension and prove that those topological pressure dimensions
have some properties similar to the topological pressure. In the end of those works,
we give the topological pressure dimension is greater than or equal to 1
for certain special almost additive sequence.

%%%%%%%%%%%%%%%%%%%%%%%%%%%%%%%%%%%%%%%%%%%%%%%%%%%%%%%%%%%%%%%%%%%%%%%%%%%%%%%%
%%%%%%%%%%%%%%%%%%%%%%%%%%%%%%%%%%%%%%%%%%%%%%%%%%%%%%%%%%%%%%%%%%%%%%%%%%%%%%%%
%%%%%%%%%%%%%%%%%%%%%%%%%%%%%%%%%%%%%%%%%%%%%%%%%%%%%%%%%%%%%%%%%%%%%%%%%%%%%%%%
%%%%%%%%%%%%%%%%%%%%%%%%%%%%%%%%%%%%%%%%%%%%%%%%%%%%%%%%%%%%%%%%%%%%%%%%%%%%%%%%

\section{Definition of topological pressure dimension for almost additive potentials}
\qquad A topological dynamical system  $(X,d,T)$ ($(X,T)$ for short) means that  $(X,d)$ is
a compact metric space together with  a continuous self-map $T:X\to X$.
Let $\mathbb{N}$ denote the set of all positive integers
and let $\mathbb{Z_+}=\mathbb{N}\cup$ $\{0\}$.
Given a topological dynamical system $(X,T)$, denote by $\mathcal{C}_{X}$ the set
of all finite open covers of $X$. Given two covers $\alpha,\beta\in\mathcal{C}_{X}$,
we say that $\beta$ is finer than $\alpha$ $(\alpha\preceq\beta)$ if for every
$V\in\beta$ there is a set $U\in\alpha$ such that $V\subseteq U$.
Let $\alpha\vee\beta=\{U\cap V:U\in\alpha,V\in\beta\}$. It is clear that
$\alpha\preceq\alpha\vee\beta$ and $\beta\preceq\alpha\vee\beta$.
Given $n\in\mathbb N$ and cover $\alpha\in\mathcal{C}_{X}$,
let
\begin{align*}
\bigvee\limits_{i=0}^{n-1}T^{-i}\alpha=\{A_{i_0}\cap T^{-1}A_{i_1}\cap\cdots\cap T^{-(n-1)}A_{i_{n-1}}
:A_{i_j}\in\alpha\}.
\end{align*}
Given $\alpha\in\mathcal{C}_X$, let $\mathcal{N}(\alpha)$ be the number of the sets in a subcover
of $\alpha$ with the smallest cardinality.

Let $n\in\mathbb{N}$ and $\epsilon>0$. Define the metric $d_{n}$  on
$X$ by
\begin{eqnarray*}
d_{n}(x,y)=\max\limits_{0\leq j<n}d(T^j(x),T^j(y)).
\end{eqnarray*}
A subset $F\subseteq X$ is an $(n,\epsilon)$-spanning set of $X$ if for
any $x\in X$, there exists $y\in F$ such that $d_n(x,y)<\epsilon$,
let $s(n,\epsilon)$ be the minimal cardinality of any
$(n,\epsilon)$-spanning set of $X$. The
dual definition is as follows.
A subset $E\subseteq X$ is an $(n,\epsilon)$-separated set of $X$ if for
any $x\neq y$ in $E$, one has $d_{n}(x,y)>\epsilon$,
let the quantity $r(n,\epsilon)$
be the maximal cardinality of $(n,\epsilon)$-separated set of $X$.
It is well known that
\begin{eqnarray*}
s(n,\epsilon)\leq r(n,\epsilon)\leq s(n,\epsilon/2).
\end{eqnarray*}

Let $C(X,\mathbb R)$ be the space of real-valued continuous functions
of $X$. For $\varphi\in C(X,\mathbb R)$ and $n\geq 1$, denote
\begin{align*}
(S_n\varphi)(x)=\sum\limits_{i=0}^{n-1}\varphi(T^i x).
\end{align*}
Now, We first introduce the class of almost additive sequence. Let $(X,T)$ be a topological dynamical system and $\Phi=(\varphi_n)_{n\in\mathbb N}$ be a
sequence of functions with $\varphi_n\in C(X,\mathbb R)$.
\begin{defn}
A sequence of functions $\Phi=(\varphi_n)_{n\in\mathbb N}$ is said to be
almost additive for $T$ if there is a constant
$C>0$ such that
\begin{align} \label{2.1}
-C+\varphi_n(x)+\varphi_m(T^n(x))\leq\varphi_{n+m}(x)\leq\varphi_n(x)+\varphi_m(T^n(x))+C
\end{align}
for every $n,m\in\mathbb N$ and $x\in X$.
\end{defn}
Clearly, given $\varphi\in C(X,\mathbb R),$ 
$(S_n\varphi)_{n\in\mathbb N}$ is an almost additive sequence. Some nontrivial
examples of almost additive sequences, related to the study of Lyapunov exponents
of nonconformal transformations, are given in
\cite{Barreira1,Barreira2,Barreira3,Barreira-Doutor1,Barreira-Doutor2}.
Next, we give the definition of topological pressure dimension
for almost additive sequences by using open covers, spanning sets
and separated sets as follows. Throughout this paper, let $(X,T)$ be a topological dynamical system and
  $\Phi=(\varphi_n)_{n\in\mathbb N}$ be an almost additive sequence for $T$ with a positive constant $C.$ 
  
Let $n\geq 1$
and $\alpha$ be an open cover of $X$. Denote
\begin{align*}
q_n(T,\Phi,\alpha)=\inf\left\{\sum\limits_{B\in\beta}\inf\limits_{x\in B}e^{\varphi_n(x)}|
\beta ~{\rm is~ a~ finite~ subcover~ of}~\bigvee\limits_{i=0}^{n-1}T^{-i}\alpha\right\}.
\end{align*}
Then we define the lower cover $s$-topological pressure with respect to $\Phi$
is to be
\begin{align*}
PD_1(s,T,\Phi)=\sup\limits_{\alpha\in\mathcal C_X}\limsup\limits_{n\to\infty}\frac{1}{n^s}\log q_n(T,\Phi,\alpha),
\end{align*}
\begin{pro}
\begin{description}
\item[(1)]
The $PD_1(s,T,\Phi)$ are nonnegative (or nonpositive) for all $s\geq 0.$
\item[(2)]
If the map $s>0\mapsto PD_1(s,T,\Phi)$ is nonnegative and decreasing
with $s$, then there exists $s_0\in[0,+\infty]$ such that
\begin{align*}
PD_1(s,T,\Phi)=\left\{
\begin{array}{ll}
 +\infty,
&\mbox{\rm if} ~ 0<s<s_0,
\\
\;
 \\
 0,&\mbox{\rm if}~s>s_0.
~~~~~~~~~~~~~~~~~~~~~~~~~~~
\end{array}
\right.
\end{align*}
\item[(3)]
If the map $s>0\mapsto PD_1(s,T,\Phi)$ is nonpositive and increasing
with $s$, then there exists $s_0\in[0,+\infty]$ such that
\begin{align*}
PD_1(s,T,\Phi)=\left\{
\begin{array}{ll}
 -\infty,
&\mbox{\rm if} ~ 0<s<s_0,
\\
\;
 \\
 0,&\mbox{\rm if}~s>s_0.
~~~~~~~~~~~~~~~~~~~~~~~~~~~
\end{array}
\right.
\end{align*}
\end{description}
 \label{pro2.1}
\end{pro}
Proposition \ref{pro2.1} (2) and (3) indicate that the value of $PD_1(s,T,\Phi)$
jumps from infinity to $0$ at both sides of some point $s_0$. Similar to the fractal
dimension, define the lower cover pressure dimension of $T$ with respect to
$\Phi$ as follows
\begin{align*}
PD_1(T,\Phi)=\sup\{s>0:PD_1(s,T,\Phi)=\infty\}=\inf\{s>0:PD_1(s,T,\Phi)=0\},
\end{align*}

\begin{rmk}
If $s_0=0$, then $PD_1(s,T,\Phi)=0$ for all $s>0$.
If $s_0=+\infty$, then $PD_1(s,T,\Phi)=+\infty$ for all $s>0$ or
$PD_1(s,T,\Phi)=-\infty$ for all $s>0$. For both critical
cases $s_0=0$ or $s_0=+\infty$, there is no jump in $PD_1(s,T,\Phi)$.
\end{rmk}

A similar definition for the upper cover pressure dimension for
almost additive sequences is as follows.

Let $\Phi=(\varphi_n)_{n\in\mathbb N}$ and $\alpha\in\mathcal C_X$. Denote
\begin{align*}
p_n(T,\Phi,\alpha)=\inf\left\{\sum\limits_{B\in\beta}\sup\limits_{x\in B}e^{\varphi_n(x)}|
\beta ~{\rm is~ a~ finite~ subcover~ of}~\bigvee\limits_{i=0}^{n-1}T^{-i}\alpha\right\},
\end{align*}
and define the upper cover $s$-topological pressure with respect $\Phi$ is to be
\begin{align*}
PD_4(s,T,\Phi)=\sup\limits_{\alpha\in\mathcal C_X}\limsup\limits_{n\to\infty}\frac{1}{n^s}\log p_n(T,\Phi,\alpha).
\end{align*}
Furthermore, using the same argument, we define the upper cover pressure dimension
of $T$ with respect to $\Phi$ as  follows:
\begin{align*}
PD_4(T,\Phi)=\sup\{s>0:PD_4(s,T,\Phi)=\infty\}=\inf\{s>0:PD_4(s,T,\Phi)=0\}.
\end{align*}
Next, we give the definitions of topological pressure dimension
for almost additive sequences by using the spanning set and the separated
set.

Let $n\geq 1, \Phi=(\varphi_n)_{n\in\mathbb N}$
and $\epsilon>0$. Denote
\begin{align*}
Q_n(T,\Phi,\epsilon)=\inf\left\{\sum\limits_{x\in F}e^{\varphi_n(x)}|
F ~{\rm is~ an~}(n,\epsilon){\rm -spanning~set~for}~X\right\}.
\end{align*}
Then we define $s$-topological pressure with respect $\Phi$ from the spanning set
to be
\begin{align*}
PD_2(s,T,\Phi)=\lim\limits_{\epsilon\to 0}
\limsup\limits_{n\to\infty}\frac{1}{n^s}\log Q_n(T,\Phi,\epsilon).
\end{align*}
By \cite{Barreira-Gelfert}, when $s=1$, $PD_2(s,T,\Phi)$ is just
the topological pressure of $T$ for almost additive sequences $\Phi$.
For the $s$-topological pressure $PD_2(s,T,\Phi)$, Proposition
\ref{pro2.1} holds as well.
Analogous to the fractal dimension, the topological pressure
dimension of $T$ with respect to $\Phi$ from spanning set is defined as follows:
\begin{align*}
PD_2(T,\Phi)=\sup\{s>0:PD_2(s,T,\Phi)=\infty\}=\inf\{s>0:PD_2(s,T,\Phi)=0\}.
\end{align*}
Denote
\begin{align*}
P_n(T,\Phi,\epsilon)=\sup\left\{\sum\limits_{x\in E}e^{\varphi_n(x)}|
E ~{\rm is~ an~}(n,\epsilon){\rm -separated~set~for}~X\right\}.
\end{align*}
Then we define $s$-topological pressure with respect $\Phi$ from the separated set
to be
\begin{align*}
PD_3(s,T,\Phi)=\lim\limits_{\epsilon\to 0}
\limsup\limits_{n\to\infty}\frac{1}{n^s}\log P_n(T,\Phi,\epsilon).
\end{align*}
We define the topological pressure dimension of $T$ with respect
to $\Phi$ from the separated set as follows:
\begin{align*}
PD_3(T,\Phi)=\sup\{s>0:PD_3(s,T,\Phi)=\infty\}=\inf\{s>0:PD_3(s,T,\phi)=0\}.
\end{align*}

Note that $Q_n(T,\Phi,\epsilon)\leq P_n(T,\phi,\epsilon)$ for any $\epsilon>0$
and $n\in\mathbb N$. Moreover, if $\alpha\in\mathcal{C}_{X}$ has a Lebesgue
number $\delta$, then $q_n(T,\Phi,\alpha)\leq Q_n(T,\Phi,\frac{\delta}{2})$
and $P_n(T,\Phi,\epsilon)\leq p_n(T,\Phi,\gamma)$ if $\epsilon>0$ and $\gamma$
is an open cover with diameter of $\gamma$ less than or equal to $\epsilon$.
This implies the following inequality:
\begin{align}\label{2.2}
PD_1(T,\Phi)\leq PD_2(T,\Phi)\leq PD_3(T,\Phi)\leq PD_4(T,\Phi).
\end{align}
\begin{rmk}
  If $\Phi=(S_n\varphi)_{n\in\mathbb N}$ for a given $\varphi\in C(X,\mathbb R),$ then $PD_i(s,T,\Phi)=PD_i(s,T,\varphi)$
   and $PD_i(T,\Phi)=PD_i(T,\varphi)$, where $PD_i(s,T,\varphi)$ and $PD_i(T,\varphi)$  are $s$-topological pressure and topological pressure dimension for potential function $\varphi$ defined by Cheng and Li \cite{Cheng-Li} respectively.
\end{rmk}

\begin{pro}
Let $(X,T)$ be a topological dynamical system and $\Phi=(\varphi_n)_{n\in\mathbb N}$ be an almost additive sequence.
If $\delta>0$ is such that $d(x,y)<\frac{\epsilon}{2}$ implies
$\mid\varphi_1(x)-\varphi_1(y)\mid<\delta$, then
\begin{align*}
P_n(T,\Phi,\epsilon)\leq e^{2nC+n\delta}Q_n(T,\Phi,\frac{\epsilon}{2}).
\end{align*}
 \label{pro2.2}
\end{pro}
\begin{proof}
Let $E$ be an $(n,\epsilon)$-separated set and $F$ be an
$(n,\frac{\epsilon}{2})$-spanning set. Define $\phi:E\to F$,
for each $x\in E$, some point $\phi(x)\in F$ with $d_n(x,\phi(x))\leq\frac{\epsilon}{2}$.
Then $\phi$ is injective, further, we have
\begin{align}\label{2.3}
\sum\limits_{y\in F}e^{\varphi_n(y)}\geq\sum\limits_{y\in\phi(E)}e^{\varphi_n(y)}
\geq(\min\limits_{x\in E}e^{\varphi_n(\phi(x))-\varphi_n(x)})\sum\limits_{x\in E}e^{\varphi_n(x)}.
\end{align}
Moreover, by inequality (\ref{2.1}), we have
\begin{align*}
-nC+\sum\limits_{i=0}^{n-1}\varphi_1(T^i(x))\leq\varphi_n(x)\leq
\sum\limits_{i=0}^{n-1}\varphi_1(T^i(x))+nC,
\end{align*}
and
\begin{align*}
-nC+\sum\limits_{i=0}^{n-1}\varphi_1(T^i(\phi(x)))\leq\varphi_n(\phi(x))\leq
\sum\limits_{i=0}^{n-1}\varphi_1(T^i(\phi(x)))+nC.
\end{align*}
Furthermore,
\begin{align*}
-2nC+\sum\limits_{i=0}^{n-1}(\varphi_1(T^i(\phi(x))-\varphi_1(T^i(x)))
&\leq\varphi_n(\phi(x))-\varphi_n(x)\\
&\leq\sum\limits_{i=0}^{n-1}(\varphi_1(T^i(\phi(x)))-\varphi_1(T^i(x)))+2nC.
\end{align*}
Since $d(x,y)<\frac{\epsilon}{2}$ implies $\mid\varphi_1(x)-\varphi_1(y)\mid<\delta$,
it follows that
\begin{align*}
-n\delta\leq\sum\limits_{i=0}^{n-1}(\varphi_1(T^i(\phi(x)))-\varphi_1(T^i(x)))\leq n\delta,
\end{align*}
i.e.,$-2nC-n\delta\leq\varphi_n(\phi(x))-\varphi_n(x)\leq n\delta+2nC$. By the inequality (\ref{2.3}),
we have
\begin{align*}
\sum\limits_{y\in F}e^{\varphi_n(y)}\geq e^{-2nC-n\delta}\sum\limits_{x\in E}e^{\varphi_n(x)}.
\end{align*}
This implies that
\begin{align*}
P_n(T,\Phi,\epsilon)\leq e^{2nC+n\delta}Q_n(T,\Phi,\frac{\epsilon}{2}).
\end{align*}
\end{proof}
\begin{rmk}
From Proposition \ref{pro2.2}, we have
\begin{align*}
\log P_n(T,\Phi,\epsilon)\leq (n\delta+2nC)+\log Q_n(T,\Phi,\frac{\epsilon}{2}),
\end{align*}
which implies
\begin{align*}
PD_3(s,T,\Phi)\leq\limsup\limits_{n\to\infty}\frac{n}{n^s}(\delta+2C)+PD_2(s,T,\Phi).
\end{align*}
This inequality holds for any $\delta$, furthermore, 
\begin{itemize}
  \item if $s=1,$ then 
  \begin{align*}
PD_3(s,T,\Phi)\leq 2C+PD_2(s,T,\Phi);
\end{align*}
  \item if $1<s<+\infty$,
then 
\begin{align*}
PD_3(s,T,\Phi)=PD_2(s,T,\Phi).
\end{align*}
\end{itemize}
\end{rmk}
Similarly, for any $\alpha\in\mathcal{C}_{X}$, if $d(x,y)<diam(\alpha)$
is such that $\mid\varphi_1(x)-\varphi_1(y)\mid\leq\delta$, where
$diam(\alpha)$ denotes the diameter of $\alpha$, then
$p_n(T,\Phi,\alpha)\leq e^{2nC+n\delta}q_n(T,\Phi,\alpha)$.
Since the above inequality holds for any $\delta>0$, it follows that for $1<s<+\infty,$
\begin{align*}
PD_1(s,T,\Phi)=PD_2(s,T,\Phi)=PD_3(s,T,\Phi)=PD_4(s,T,\Phi).
\end{align*}
%%%%%%%%%%%%%%%%%%%%%%%%%%%%%%%%%%%%%%%%%%%%%%%%%%%%%%%%%%%%%%%%%%%%%%%%%%%%%%%%
%%%%%%%%%%%%%%%%%%%%%%%%%%%%%%%%%%%%%%%%%%%%%%%%%%%%%%%%%%%%%%%%%%%%%%%%%%%%%%%%
%%%%%%%%%%%%%%%%%%%%%%%%%%%%%%%%%%%%%%%%%%%%%%%%%%%%%%%%%%%%%%%%%%%%%%%%%%%%%%%%
%%%%%%%%%%%%%%%%%%%%%%%%%%%%%%%%%%%%%%%%%%%%%%%%%%%%%%%%%%%%%%%%%%%%%%%%%%%%%%%%
\section{ $s$-topological pressure for almost additive potentials}
 Let $(X,d,T)$ be a topological dynamical system. Let $n\in\mathbb N$, $\epsilon>0$
and $s>0$. Carvalho \cite{Carvalho} and Cheng and Li \cite{Cheng-Li} gave the  definitions of $s$-topological
entropy and topological entropy dimension as follows:

\noindent $s$-topological entropy is to be as
\begin{align*}
D(s,T)=\lim\limits_{\epsilon\to 0}\limsup\limits_{n\to\infty}
\frac{1}{n^s}\log s(n,\epsilon)=\lim\limits_{\epsilon\to 0}\limsup\limits_{n\to\infty}
\frac{1}{n^s}\log r(n,\epsilon).
\end{align*}
And the topological entropy dimension of $T$ is to be
\begin{align*}
D(T)=\sup\{s>0:D(s,T)=\infty\}=\inf\{s>0:D(s,T)=0\}.
\end{align*}
\begin{pro}\label{pro3.1}
Let $\Phi=(\varphi_n)_{n\in\mathbb N}$ and $\Psi=(\psi_n)_{n\in\mathbb N}$
be two almost additive sequences. Define
\begin{align*}
\Phi+\Psi=(\varphi_n+\psi_n)_{n\in\mathbb N}.
\end{align*}
Then $\Phi+\Psi$ is an almost additive sequence.
\end{pro}
\begin{proof}
Since $\Phi=(\varphi_n)_{n\in\mathbb N}$ and $\Psi=(\psi_n)_{n\in\mathbb N}$
are two almost additive sequences, there exist $C_1>0$ and $C_2>0$ such that
\begin{align}\label{3.4}
-C_1+\varphi_n(x)+\varphi_m(T^n(x))\leq\varphi_{n+m}(x)\leq\varphi_n(x)+\varphi_m(T^n(x))+ C_1
\end{align}
and
\begin{align}\label{3.5}
-C_2+\psi_n(x)+\psi_m(T^n(x))\leq\psi_{n+m}(x)\leq\psi_n(x)+\psi_m(T^n(x))+C_2
\end{align}
for every $n,m\in\mathbb N$ and $x\in X$. By the inequalities (\ref{3.4}) and (\ref{3.5}),
we have
\begin{align*}
&-(C_1+C_2)+(\varphi_n(x)+\psi_n(x))+(\varphi_m(T^n(x))+\psi_m(T^n(x)))\\
&\leq(\varphi_{n+m}(x)+\psi_{n+m}(x))\\
&\leq (\varphi_n(x)+\psi_n(x))+(\varphi_m(T^n(x))+\psi_m(T^n(x)))+(C_1+C_2),
\end{align*}
which implies that
\begin{align*}
&-(C_1+C_2)+(\varphi_n+\psi_n)(x)+(\varphi_m+\psi_m)(T^n(x))\\
&\leq(\varphi_{n+m}+\psi_{n+m})(x)\\
&\leq (\varphi_n+\psi_n)(x)+(\varphi_m+\psi_m)(T^n(x))+(C_1+C_2).
\end{align*}
This shows that $\Phi+\Psi=(\varphi_n+\psi_n)_{n\in\mathbb N}$ is an almost additive sequence.
\end{proof}

Similarly, we can easily prove that 
if $\Phi=(\varphi_n)_{n\in\mathbb N}$ is an almost additive sequence and $\lambda\in\mathbb R,$ so is $\lambda\Phi=(\lambda\varphi_n)_{n\in\mathbb N}$.
Denote $\Phi=\mathbf{0}$ if for every $n\in\mathbb N$ and $x\in X$, $\varphi_n(x)=0$ and $\Phi\leq\Psi$
if for every $n\in\mathbb N$ and $x\in X$, $\varphi_n(x)\leq\psi_n(x)$.

\begin{thm}\label{thm3.1}
Let $(X,T)$ be a topological dynamical system and $\Phi=(\varphi_n)_{n\in\mathbb N}$
be an almost additive sequence.
\begin{description}
\item[(1)]
$PD_i(s,T,\mathbf{0})=D(s,T)$, $i=1,2,3,4.$
\item[(2)]
If $s=1$, then
\begin{align*}
-C+D(s,T)+\inf\varphi_1\leq PD_i(s,T,\Phi)\leq D(s,T)+\sup\varphi_1+C,~
i=1,2,3,4.
\end{align*}
\item[(3)]
If $s>1$, then $PD_i(s,T,\Phi)=D(s,T)$, $i=1,2,3,4$.
\end{description}
\end{thm}

\begin{proof}
(1) Since $\Phi=\mathbf{0}$, i.e., $\varphi_n=0$ for every $n\in\mathbb N$, we have
$\sum\limits_{B\in\beta}\inf\limits_{x\in B}e^{\varphi_n(x)}$ and
$\sum\limits_{B\in\beta}\sup\limits_{x\in B}e^{\varphi_n(x)}$ become
the cardinality of the subcover $\beta$, and $\sum\limits_{x\in F}e^{\varphi_n(x)}$
and $\sum\limits_{x\in E}e^{\varphi_n(x)}$ become the cardinalities of the
spanning set $F$ and the separated set $E$ respectively. Therefore, according
to the definitions of the $D(s,T)$ and $PD_i(s,T,\Phi)$, we have
\begin{align*}
PD_i(s,T,\mathbf{0})=D(s,T), ~i=1,2,3,4.
\end{align*}

(2) Since $\Phi=(\varphi_n)_{n\in\mathbb N}$ is an almost additive sequence,
by the inequality (\ref{2.1}), we have
\begin{align*}
-nC+\sum\limits_{i=0}^{n-1}\varphi_1(T^i(x))\leq\varphi_n(x)\leq
\sum\limits_{i=0}^{n-1}\varphi_1(T^i(x))+nC.
\end{align*}
Furthermore,
\begin{align*}
-nC+n\inf\varphi_1\leq\varphi_n(x)\leq n\sup\varphi_1+nC,
\end{align*}
and $\inf\varphi_1,\sup\varphi_1$ are finite since $X$ is compact and $\varphi_1$
is continuous. Hence, we have
\begin{align}\label{3.6}
e^{-nC+n\inf\varphi_1}\leq e^{\varphi_n(x)}\leq e^{n\sup\varphi_1+nC}.
\end{align}
By $s=1$ and the definitions of $PD_i(s,T,\Phi)$, we have
\begin{align*}
-C+D(s,T)+\inf\varphi_1\leq PD_i(s,T,\Phi)\leq D(s,T)+\sup\varphi_1+C,~
i=1,2,3,4.
\end{align*}

(3) If $s>1$, then by the above inequality (\ref{3.6}) and the definitions of $D(s,T)$ and $PD_i(s,T,\Phi)$,
we have
 $D(s,T)=PD_i(s,T,\Phi)$, $i=1,2,3,4$.
\end{proof}

\begin{thm}\label{thm3.2}
Let $(X,T)$ be a topological dynamical system and $s>0$. Let
$\Phi=(\varphi_n)_{n\in\mathbb N}$ and $\Psi=(\psi_n)_{n\in\mathbb N}$
be two almost additive sequences. Then
\begin{description}
\item[(1)]
$PD_3(s,T,\Phi+\Psi)\leq PD_3(s,T,\Phi)+PD_3(s,T,\Psi)$.
\item[(2)]
For $ i=1,2,3,4.$
\begin{eqnarray*}
PD_i(s,T,\lambda\Phi)\left\{
\begin{array}{ll}
 \leq\lambda\cdot PD_i(s,T,\Phi),
&\mbox{\rm if} ~ \lambda\geq 1,
\\
\;
 \\
 \geq\lambda\cdot PD_i(s,T,\Phi),&\mbox{\rm if}~0\leq\lambda\leq 1.

~~~~~~~~~~~~~~~~~~~~~~~~~~~
\end{array}
\right.
\end{eqnarray*}
\end{description}
\end{thm}

\begin{proof}
(1) By Proposition \ref{pro3.1}, $\Phi+\Psi=(\varphi_n+\psi_n)_{n\in\mathbb N}$ is an
almost additive sequence. Let $E$ be an $(n,\epsilon)$-separated set of $X$ for $T$.
Since
\begin{align*}
\sum\limits_{x\in E}e^{(\varphi_n+\psi_n)(x)}\leq(\sum\limits_{x\in E}e^{\varphi_n(x)})
(\sum\limits_{x\in E}e^{\psi_n(x)}),
\end{align*}
we have $P_n(T,\Phi+\Psi,\epsilon)\leq P_n(T,\Phi,\epsilon)P_n(T,\Psi,\epsilon)$,
which implies
\begin{align*}
PD_3(s,T,\Phi+\Psi)\leq PD_3(s,T,\Phi)+PD_3(s,T,\Psi).
\end{align*}

(2) We only prove that $i=1$ holds. Let $\alpha\in\mathcal{C}_X$ and
$\beta$ is a subcover of $\bigvee\limits_{i=0}^{n-1}T^{-i}(\alpha)$.
We need the following results to finish the proof. If $a_1,a_2,\cdots,a_k$ are
positive numbers with $\sum\limits_{i=1}^ka_i=1$, then
\begin{itemize}
   \item $\sum\limits_{i=1}^ka_i^{\lambda}\leq 1$ when $\lambda\geq 1$,
   \item $\sum\limits_{i=1}^ka_i^{\lambda}\geq 1$ when $0\leq\lambda\leq 1$.
\end{itemize}
Note that $\inf\limits_{x\in B}e^{(\lambda\varphi_n)(x)}=(\inf\limits_{x\in B}e^{\varphi_n(x)})^{\lambda}$
for $\lambda\geq 0$ and $B\in\beta$, then
\begin{align*}
\sum\limits_{B\in\beta}\inf\limits_{x\in B}e^{(\lambda\varphi_n)(x)}
\leq(\sum\limits_{B\in\beta}\inf\limits_{x\in B}e^{\varphi_n(x)})^{\lambda}~{\rm if}~\lambda\geq 1
\end{align*}
and
\begin{align*}
\sum\limits_{B\in\beta}\inf\limits_{x\in B}e^{(\lambda\varphi_n)(x)}
\geq(\sum\limits_{B\in\beta}\inf\limits_{x\in B}e^{\varphi_n(x)})^{\lambda}~{\rm if}~0\leq\lambda\leq 1.
\end{align*}
Therefore, we have
\begin{align*}
q_n(T,\lambda\Phi,\alpha)\leq (q_n(T,\Phi,\alpha))^{\lambda}~{\rm if}~\lambda\geq 1
\end{align*}
and
\begin{align*}
q_n(T,\lambda\Phi,\alpha)\geq (q_n(T,\Phi,\alpha))^{\lambda}~{\rm if}~0\leq\lambda\leq 1,
\end{align*}
which implies (2) holds for $i=1$. Similarly, we can prove that (2) holds for
$i=2,3,4.$
\end{proof}

\begin{thm}\label{thm3.3}
Let $(X,T)$ be a topological dynamical system and $s>0$. Let
$\Phi=(\varphi_n)_{n\in\mathbb N}$ and $\Psi=(\psi_n)_{n\in\mathbb N}$
be two almost additive sequences.
\begin{description}
\item[(1)]
If $\Phi\leq\Psi$, then
$PD_i(s,T,\Phi)\leq PD_i(s,T,\Psi)$, $i=1,2,3,4$.
\item[(2)]
If $s>1$, then
$PD_i(s,T,\Phi+\Psi\circ T-\Psi)=PD_i(s,T,\Phi)$, $i=1,2,3,4.$
\item[(3)]
$PD_3(s,T,\cdot)$ is convex.
\end{description}
\end{thm}

\begin{proof}
(1) Since $\Phi\leq\Psi$, we have $\varphi_n(x)\leq\psi_n(x)$
for any $x\in X$ and $n\in\mathbb N$. Furthermore,
$e^{\varphi_n(x)}\leq e^{\psi_n(x)}$. By the definitions of
$PD_i(s,T,\Phi)$, we have
\begin{align*}
PD_i(s,T,\Phi)\leq PD_i(s,T,\Psi),~ i=1,2,3,4.
\end{align*}

(2) We only prove that $PD_1(s,T,\Phi+\Psi\circ T-\Psi)=PD_1(s,T,\Phi)$.
Let $\alpha\in\mathcal{C}_X$ and
$\beta$ be a subcover of $\bigvee\limits_{i=0}^{n-1}T^{-i}(\alpha)$.
We can easily prove that
$\Phi+\Psi\circ T-\Psi=(\varphi_n+\psi_n\circ T-\psi_n)_{n\in\mathbb N}$
is an almost additive sequence.
Since $\Psi=(\psi_n)_{n\in\mathbb N}$
is an almost additive sequence, there exists $C>0$ such that
\begin{align*}
-nC+\sum\limits_{i=0}^{n-1}\psi_1(T^i(x))\leq\psi_n(x)\leq
\sum\limits_{i=0}^{n-1}\psi_1(T^i(x))+nC
\end{align*}
and
\begin{align*}
-nC+\sum\limits_{i=1}^{n}\psi_1(T^i(x))\leq\psi_n\circ T(x)\leq
\sum\limits_{i=1}^{n}\psi_1(T^i(x))+nC.
\end{align*}
Furthermore, we have
\begin{align}\label{3.7}
-2nC+\psi_1(T^n(x))-\psi_1(x)\leq\psi_n\circ T(x)-\psi_n(x)\leq\psi_1(T^n(x))-\psi_1(x)+2nC.
\end{align}
Denote $\parallel\varphi\parallel=\max\limits_{x\in X}|\varphi(x)|$. Hence, by the inequality (\ref{3.7}),
we have
\begin{align*}
-2nC-2\parallel\psi_1\parallel\leq \psi_n\circ T(x)-\psi_n(x)\leq2\parallel\psi_1\parallel+2nC.
\end{align*}
 Note that
\begin{align*}
\sum\limits_{B\in\beta}\inf\limits_{x\in B}e^{\varphi_n(x)+(-2nC+\psi_1(T^n(x))-\psi_1(x))}
&\leq
\sum\limits_{B\in\beta}\inf\limits_{x\in B}e^{\varphi_n(x)+\psi_n\circ T(x)-\psi_n(x)}\\
&\leq
\sum\limits_{B\in\beta}\inf\limits_{x\in B}e^{\varphi_n(x)+(\psi_1(T^n(x))-\psi_1(x)+2nC)}
\end{align*}
Therefore, we have
\begin{align*}
e^{-2nC-2\parallel\psi_1\parallel}\sum\limits_{B\in\beta}\inf\limits_{x\in B}e^{\varphi_n(x)}
\leq\sum\limits_{B\in\beta}\inf\limits_{x\in B}e^{\varphi_n(x)+\psi_n\circ T(x)-\psi_n(x)}
\leq e^{2nC+2\parallel\psi_1\parallel}\sum\limits_{B\in\beta}\inf\limits_{x\in B}e^{\varphi_n(x)},
\end{align*}
which implies
\begin{align*}
e^{-2nC-2\parallel\psi_1\parallel}q_n(T,\Phi,\alpha)
\leq q_n(T,\Phi+\Psi\circ T-\Psi,\alpha)
\leq e^{2nC+2\parallel\psi_1\parallel}q_n(T,\Phi,\alpha)
\end{align*}
for any $\alpha\in\mathcal{C}_{X}$. Furthermore,
\begin{align*}
&(-2nC-2\parallel\psi_1\parallel)+\log q_n(T,\Phi,\alpha)\\
&\leq\log q_n(T,\Phi+\Psi\circ T-\Psi,\alpha)\\
&\leq (2nC+2\parallel\psi_1\parallel)+\log q_n(T,\Phi,\alpha).
\end{align*}
Therefore, by the definition of $PD_1(s,T,\Phi)$ and $s>1$, we have
\begin{align*}
PD_1(s,T,\Phi+\Psi\circ T-\Psi)=PD_1(s,T,\Phi).
\end{align*}
Similarly, we can prove $PD_i(s,T,\Phi+\Psi\circ T-\Psi)=PD_i(s,T,\Phi)(i=2,3,4)$
for $s>1$.

(3) Let $t\in [0,1]$ and $E$ be an $(n,\epsilon)$-separated set of $X$ for $T$.
By H\"{o}lder's inequality, we have
\begin{align*}
\sum\limits_{x\in E}e^{t\varphi_n(x)+(1-t)\psi_n(x)}\leq (\sum\limits_{x\in E}e^{\varphi_n(x)})^t
(\sum\limits_{x\in E}e^{\phi_n(x)})^{1-t}.
\end{align*}
This implies
\begin{align*}
P_n(T,t\Phi+(1-t)\Psi,\epsilon)\leq (P_n(T,\Phi,\epsilon))^t(P_n(T,\Psi,\epsilon))^{1-t}.
\end{align*}
Therefore,
\begin{align*}
PD_3(T,t\Phi+(1-t)\Psi)\leq tPD_3(s,T,\Phi)+(1-t)PD_3(s,T,\Psi).
\end{align*}
\end{proof}
Let $(X,T)$ be a topological dynamical system and $\Phi=(\varphi_n)_{n\in\mathbb N}$ be an almost additive sequence.
Define $\Phi_k=(\varphi_{nk})_{n\in\mathbb N}$. Clearly, $\Phi_k$ is
also an almost additive sequence. 

\begin{pro}\label{pro3.2}
Let $\Phi=(\varphi_n)_{n\in\mathbb N}$ be an almost additive sequence(with respect to $T$ in $X$).
If $T$ is a homeomorphism, then
$\Phi'=(\varphi_n\circ T^{-(n-1)})_{n\in\mathbb N}$ is an almost additive sequence(with respect to $T^{-1}$ in $X$).
\end{pro}

\begin{proof}
Denote $\varphi_n'=\varphi_n\circ T^{-(n-1)}$. Since $\Phi=(\varphi_n)_{n\in\mathbb N}$,
there exists $C>0$ such that
\begin{align}\label{3.8}
-C+\varphi_n(x)+\varphi_m(T^n(x))\leq\varphi_{n+m}(x)\leq\varphi_n(x)+\varphi_m(T^n(x))+C
\end{align}
for every $n,m\in\mathbb N$ and $x\in X$. Next, we prove
\begin{align}\label{3.9}
-C+\varphi_m'(x)+\varphi_n'(T^{-m}(x))\leq\varphi_{n+m}'(x)\leq\varphi_m'(x)+\varphi_n'(T^{-m}(x))+C.
\end{align}
In the inequality (\ref{3.8}), by using $T^{-(n+m-1)}(x)$ instead of $x$, then
\begin{align*}
&-C+\varphi_n(T^{-(n+m-1)}(x))+\varphi_m(T^n(T^{-(n+m-1)}(x)))\\
&\leq\varphi_{n+m}(T^{-(n+m-1)}(x))\\
&\leq\varphi_n(T^{-(n+m-1)}(x))+\varphi_m(T^n(T^{-(n+m-1)}(x)))+C
\end{align*}
i.e.,
\begin{align*}
&-C+(\varphi_n\circ T^{-(n-1)}(T^{-m}(x))+(\varphi_m\circ T^{-(m-1)})(x)\\
&\leq(\varphi_{n+m}\circ T^{-(n+m-1)})(x)\\
&\leq(\varphi_n\circ T^{-(n-1)}(T^{-m}(x))+(\varphi_m\circ T^{-(m-1)})(x)+C,
\end{align*}
which implies the inequality (\ref{3.9}) holds. Therefore,
$\Phi'=(\varphi_n\circ T^{-(n-1)})_{n\in\mathbb N}$ is an almost additive sequence for $T^{-1}$.
\end{proof}

\begin{thm}\label{thm3.4}
Let $(X,T)$ be a topological dynamical system and $s>0$.
Let $\Phi=(\varphi_n)_{n\in\mathbb N}$ be an almost additive sequence for $T$
and $\Phi'=(\varphi_n\circ T^{-(n-1)})_{n\in\mathbb N}$.
\begin{description}
\item[(1)]
If $k>0$, then
\begin{align*}
PD_i(s,T^k,\Phi_k)\leq k^sPD_i(s,T,\Phi),~i=1,2,3,4.
\end{align*}
\item[(2)]
\begin{itemize}
\item If $s=1,$
\begin{align*}
-C+PD_i(s,T,\varphi_1)\leq PD_i(s,T,\Phi)\leq PD_i(s,T,\varphi_1)+C,~ i=1,2,3,4;
\end{align*}
\item if $s>1,$
\begin{align*}
PD_i(s,T,\varphi_1)= PD_i(s,T,\Phi),~ i=1,2,3,4.
\end{align*}
\end{itemize}
\item[(3)]
If $T$ is a homeomorphism, then
\begin{align*}
PD_i(s,T^{-1},\Phi')=PD_i(s,T,\Phi), ~i=1,2,3,4.
\end{align*}
\end{description}
\end{thm}

\begin{proof}
(1) Denote
\begin{align*}
d_{nk}^T(x,y)=\max\limits_{0\leq i\leq nk-1}d(T^i(x),T^i(y))
\end{align*}
and
\begin{align*}
d_{n}^{T^k}(x,y)=\max\limits_{0\leq i\leq n-1}d(T^{ki}(x),T^{ki}(y)),
\end{align*}
where $x,y\in X$. Then $d_{nk}^T(x,y)\geq d_{n}^{T^k}(x,y)$.

If $F$ is an $(nk,\epsilon)$-spanning set of $X$ for $T$, then $F$
is an $(n,\epsilon)$-spanning set of $X$ for $T^k$. Since
\begin{align*}
Q_{nk}(T,\Phi,\epsilon)=\inf\left\{\sum\limits_{x\in F}e^{\varphi_{nk}(x)}|
F ~{\rm is~ an~}(nk,\epsilon){\rm -spanning~set~of}~X~{\rm for~}T\right\}
\end{align*}
and
\begin{align*}
Q_n(T^k,\Phi_k,\epsilon)=\inf\left\{\sum\limits_{x\in F}e^{\varphi_{nk}(x)}|
F ~{\rm is~ an~}(n,\epsilon){\rm -spanning~set~of}~X~{\rm for~}T^k\right\},
\end{align*}
it follows that $Q_n(T^k,\Phi_k,\epsilon)\leq Q_{nk}(T,\Phi,\epsilon)$.
Furthermore, we have
\begin{align*}
\lim\limits_{\epsilon\to 0}
\limsup\limits_{n\to\infty}\frac{1}{n^s}\log Q_n(T^k,\Phi_k,\epsilon)
\leq k^s\lim\limits_{\epsilon\to 0}
\limsup\limits_{n\to\infty}\frac{1}{(nk)^s}\log Q_{nk}(T,\Phi,\epsilon).
\end{align*}
Therefore, $PD_2(s,T^k,\Phi_k)\leq k^sPD_2(s,T,\Phi)$.

If $E$ is an $(n,\epsilon)$-separated set of $X$ for $T^k$, then $E$ is
an $(nk,\epsilon)$-separated set of $X$ for $T$. Since
\begin{align*}
P_{nk}(T,\Phi,\epsilon)=\sup\left\{\sum\limits_{x\in E}e^{\varphi_{nk}(x)}|
E ~{\rm is~ an~}(n,\epsilon){\rm -separated~set~of}~X~{\rm for~}T\right\}
\end{align*}
and
\begin{align*}
P_n(T^k,\Phi_k,\epsilon)=\sup\left\{\sum\limits_{x\in E}e^{\varphi_{nk}(x)}|
E ~{\rm is~ an~}(n,\epsilon){\rm -separated~set~of}~X~{\rm for~}T^k\right\},
\end{align*}
it follows that $P_n(T^k,\Phi_k,\epsilon)\leq P_{nk}(T,\Phi,\epsilon)$.
Therefore, $PD_3(s,T^k,\Phi_k)\leq k^sPD_3(s,T,\Phi)$.

Let $\alpha\in\mathcal{C}_X$. Then $\bigvee\limits_{i=0}^{nk-1}T^{-i}\alpha$
is finer than $\bigvee\limits_{i=0}^{n-1}T^{-ki}\alpha$. Since
\begin{align*}
q_n(T^k,\Phi_k,\alpha)=\inf\left\{\sum\limits_{B\in\beta}\inf\limits_{x\in B}e^{\varphi_{nk}(x)}|
\beta ~{\rm is~ a~ finite~ subcover~ of}~\bigvee\limits_{i=0}^{n-1}T^{-ki}\alpha\right\}
\end{align*}
and
\begin{align*}
q_{nk}(T,\Phi,\alpha)=\inf\left\{\sum\limits_{B\in\beta}\inf\limits_{x\in B}e^{\varphi_{nk}(x)}|
\beta ~{\rm is~ a~ finite~ subcover~ of}~\bigvee\limits_{i=0}^{nk-1}T^{-i}\alpha\right\},
\end{align*}
it follows that $q_n(T^k,\Phi_k,\alpha)\leq q_{nk}(T,\Phi,\alpha)$. Therefore,
$PD_1(s,T^k,\Phi_k)\leq k^sPD_1(s,T,\Phi)$.
Similarly, we can prove $PD_4(s,T^k,\Phi_k)\leq k^sPD_4(s,T,\Phi)$.

(2) Since $\Phi=(\varphi_n)_{n\in\mathbb N}$ is an almost additive sequence
for $T$, there exists $C>0$ such that
\begin{align*}
-C+\varphi_n(x)+\varphi_m(T^n(x))\leq\varphi_{n+m}(x)\leq\varphi_n(x)+\varphi_m(T^n(x))+C
\end{align*}
for every $n,m\in\mathbb N$ and $x\in X$. Furthermore, we have
\begin{align*}
-nC+\sum\limits_{i=0}^{n-1}\varphi_1(T^i(x))\leq\varphi_n(x)\leq\sum\limits_{i=0}^{n-1}\varphi_1(T^i(x))+nC,
\end{align*}
i.e., $-nC+(S_n\varphi_1)(x)\leq\varphi_n(x)\leq (S_n\varphi_1)(x)+nC$. Therefore,
\begin{align}\label{3.10}
e^{-nC+(S_n\varphi_1)(x)}\leq e^{\varphi_n(x)}\leq e^{(S_n\varphi_1)(x)+nC}.
\end{align}
This implies that 
\begin{itemize}
\item if $s=1,$ 
\begin{align*}
-C+PD_i(s,T,\varphi_1)\leq PD_i(s,T,\Phi)\leq PD_i(s,T,\varphi_1)+C,~ i=1,2,3,4;
\end{align*}
\item if $s>1,$
\begin{align*}
PD_i(s,T,\varphi_1)= PD_i(s,T,\Phi),~ i=1,2,3,4.
\end{align*}
\end{itemize}

(3) Denote $\Phi'=(\varphi_n')_{n\in\mathbb N}$, where $\varphi_n'=\varphi_n\circ T^{-(n-1)}$.
Since
\begin{align*}
d_n(x,y)=\max\limits_{0\leq i<n}d(T^i(x),T^i(y))
=\max\limits_{0\leq i<n}d(T^{-i}(T^{n-1}(x)),T^{-i}(T^{n-1}(y)))
\end{align*}
for every $x,y\in X$. Hence, $E$ is an $(n,\epsilon)$-separated for $T$ if and only if
$T^{n-1}(E)$ is an $(n,\epsilon)$-separated set for $T^{-1}$. Note that
\begin{align*}
\sum\limits_{x\in E}e^{\varphi_n(x)}=\sum\limits_{y\in T^{n-1}(E)}e^{(\varphi_n\circ T^{-(n-1)})(y)}
=\sum\limits_{y\in T^{n-1}(E)}e^{\varphi_n'(y)}.
\end{align*}
Therefore, $Q_n(T^{-1},\Phi',\epsilon)=Q_n(T,\Phi,\epsilon)$, which implies
$PD_2(s,T^{-1},\Phi')=PD_2(s,T,\Phi)$. Similarly, we can prove
$PD_3(s,T^{-1},\Phi')=PD_3(s,T,\Phi)$.

Let $\alpha\in\mathcal{C}_X$ and $\beta$ be a subcover of $\bigvee\limits_{i=0}^{n-1}T^{-i}\alpha$.
Let $B\in\beta$. Note that
\begin{align*}
\inf\limits_{x\in B}e^{\varphi_n(x)}=\inf\limits_{y\in T^{n-1}(B)}e^{(\varphi_n\circ T^{-(n-1)})(y)}
=\inf\limits_{y\in T^{n-1}(B)}e^{\varphi_n'(y)}.
\end{align*}
Since
\begin{align*}
\bigvee\limits_{i=0}^{n-1}(T^{-1})^{-i}\alpha=\bigvee\limits_{i=0}^{n-1}T^{-i+(n-1)}\alpha
=T^{n-1}(\bigvee\limits_{i=0}^{n-1}T^{-i}\alpha),
\end{align*}
it follows that
\begin{align*}
q_n(T^{-1},\Phi',\alpha)&=\inf\left\{\sum\limits_{B\in\beta}\inf\limits_{x\in B}e^{\varphi_n'(x)}|
\beta ~{\rm is~ a~ finite~ subcover~ of}~\bigvee\limits_{i=0}^{n-1}(T^{-1})^{-i}\alpha\right\}\\
&=\inf\left\{\sum\limits_{B\in\beta}\inf\limits_{y\in T^{n-1}(B)}e^{\varphi_n(y)}|
\beta ~{\rm is~ a~ finite~ subcover~ of}~\bigvee\limits_{i=0}^{n-1}T^{-i}\alpha\right\}.
\end{align*}
The above equalities imply that
\begin{align*}
q_n(T^{-1},\Phi',\alpha)=\inf\left\{\sum\limits_{B\in\beta}\inf\limits_{x\in B}e^{\varphi_n(x)}|
\beta ~{\rm is~ a~ finite~ subcover~ of}~\bigvee\limits_{i=0}^{n-1}T^{-i}\alpha\right\}=q_n(T,\Phi,\alpha),
\end{align*}
which shows $PD_1(s,T^{-1},\Phi')=PD_1(s,T,\Phi)$. The case $i=4$ can be proved similarly.
\end{proof}
Let $(X,T_1)$ and $(Y,T_2)$ be two topological dynamical systems. Then,
$(X,T_1)$ is an extension of $(Y,T_2)$, or $(Y,T_2)$ is a factor of
$(X,T_1)$ if there exists a surjective continuous map $\pi: X\to Y$
(called a factor map) such that $\pi\circ T_1(x)=T_2\circ\pi(x)$ for
every $x\in X$. If $\pi$ is a homeomorphism, then $(X,T_1)$
and $(Y,T_2)$ are said to be topologically conjugate and the
homeomorphism $\pi$ is called a conjugate map.

\begin{thm}\label{thm3.5}
Let $(X,d_1,T_1)$ and $(Y,d_2,T_2)$ be two topological dynamical systems.
Let $\Phi=(\varphi_n)_{n\in\mathbb N}$ be an almost additive sequence
for $T_2$.
If $(Y,T_2)$ is a factor of $(X,T_1)$ with a factor map $\pi:X\to Y$, then
$PD_2(s,T_1,\Phi\circ\pi)\geq PD_2(s,T_2,\Phi)$ for any $s>0$, where
$\Phi\circ\pi=(\varphi_n\circ\pi)_{n\in\mathbb N}$. Moreover,
If $(X,T_1)$ and $(Y,T_2)$ are topologically conjugate with a conjugate
map $\pi:X\to Y$, then $PD_2(s,T_1,\Phi\circ\pi)=PD_2(s,T_2,\Phi)$
for any $s>0$.
\end{thm}

\begin{proof}
If $(Y,T_2)$ is a factor of $(X,T_1)$ with a factor map $\pi:X\to Y$,
then $\pi\circ T_1(x)=T_2\circ\pi(x)$ for every $x\in X$. We first
prove that if $\Phi=(\varphi_n)_{n\in\mathbb N}$ is
an almost additive sequence for $T_2$, then
$\Phi\circ\pi=(\varphi_n\circ\pi)_{n\in\mathbb N}$ is
an almost additive sequence for $T_1$.  In fact, since $\Phi=(\varphi_n)_{n\in\mathbb N}$ is
an almost additive sequence with respect to $T_2$   there exists $C>0$ such that
\begin{align*}
-C+\varphi_n(y)+\varphi_m(T_2^n(y))\leq\varphi_{n+m}(y)\leq\varphi_n(y)+\varphi_m(T_2^n(y))+C
\end{align*}
for every $n,m\in\mathbb N$  and $y\in Y$. Note that for any $x\in X,$ there exists $y\in Y$
such that $\pi(x)=y.$ Furthermore,
\begin{align*}
-C +\varphi_n(\pi(x))+\varphi_m(T_2^n(\pi(x)))\leq
\varphi_{n+m}(\pi(x))\leq\varphi_n(\pi(x))+\varphi_m(T_2^n(\pi(x)))+C.
\end{align*}
Therefore, we have
\begin{align*}
-C +\varphi_n\circ\pi(x) +\varphi_m\circ\pi(T_1^n(x)) \leq
\varphi_{n+m}\circ\pi(x) \leq\varphi_n\circ\pi(x) +\varphi_m\circ\pi(T_1^n(x))+C ,
\end{align*}
which implies $\Phi\circ\pi=(\varphi_n\circ\pi)_{n\in\mathbb N}$ is an almost
additive sequence for $T_1$.

\noindent Since $X $ is compact and $\pi$ is a continuous map, then for any
given $\epsilon>0$, there exists $\delta>0$ such that $d_1(x,y)<\delta$
satisfying $d_2(\pi(x),\pi(y))<\epsilon$. Combining this and  $\pi$ is surjective, if $F$ is
an $(n,\delta)$-spanning set of $X$ for $T_1$, then $\pi(F)$ is
an $(n,\epsilon)$-spanning set of $Y$ for $T_2$. Thus, we have
\begin{align*}
\sum\limits_{x\in F}e^{(\varphi_n\circ\pi)(x)}=
\sum\limits_{y\in\pi(F)}e^{\varphi_n(y)},
\end{align*}
which implies that $Q_n(T_1,\Phi\circ\pi,\delta)\geq Q_n(T_2,\Phi,\epsilon)$. Therefore,
$PD_2(s,T_1,\Phi\circ\pi)\geq PD_2(s,T_2,\Phi).$

If $\pi$ is a conjugate map, then we can use $T_2,T_1,\pi^{-1},\Phi\circ\pi$
instead of the above with $T_1,T_2,\pi,\Phi$ respectively, which implies that
$PD_2(s,T_2,\Phi)\geq PD_2(s,T_1,\Phi\circ\pi)$. Therefore,
$PD_2(s,T_1,\Phi\circ\pi)=PD_2(s,T_2,\Phi).$
 
\end{proof}

%%%%%%%%%%%%%%%%%%%%%%%%%%%%%%%%%%%%%%%%%%%%%%%%%%%%%%%%%%%%%%%%%%%%%%%%%%%%%%%%
%%%%%%%%%%%%%%%%%%%%%%%%%%%%%%%%%%%%%%%%%%%%%%%%%%%%%%%%%%%%%%%%%%%%%%%%%%%%%%%%
%%%%%%%%%%%%%%%%%%%%%%%%%%%%%%%%%%%%%%%%%%%%%%%%%%%%%%%%%%%%%%%%%%%%%%%%%%%%%%%%
%%%%%%%%%%%%%%%%%%%%%%%%%%%%%%%%%%%%%%%%%%%%%%%%%%%%%%%%%%%%%%%%%%%%%%%%%%%%%%%%
\section{Topological pressure dimension for almost additive potentials}

%%%%%%%%%%%%%%%%%%%%%%%%%%%%%%%%%%%%%%%%%%%%%%%%%%%%%%%%%%%%%%%%%%%%%%%%%
%%%%%%%%%%%%%%%%%%%%%%%%%%%%%%%%%%%%%%%%%%%%%%%%%%%%%%%%%%%%%%%%%%%%%%%%%
%%%%%%%%%%%%%%%%%%%%%%%%%%%%%%%%%%%%%%%%%%%%%%%%%%%%%%%%%%%%%%%%%%%%%%%%%
%%%%%%%%%%%%%%%%%%%%%%%%%%%%%%%%%%%%%%%%%%%%%%%%%%%%%%%%%%%%%%%%%%%%%%%%%
\qquad In this section, we discuss the relation between $D(T)$ and
$PD_i(T,\Phi)$, $i=1,2,3,4$ for certain special almost additive
sequence $\Phi$. Moreover, we give some examples on topological
pressure dimension for almost additive sequences.

\begin{thm}\label{thm4.1}
Let $(X,T)$ be a topological dynamical system and $\Phi=(\varphi_n)_{n\in\mathbb N}$
be an almost additive sequence. If $D(T)<1$, then $PD_i(T,\Phi)\leq 1$, $i=1,2,3,4$.
\end{thm}

\begin{proof}
Since $D(T)<1$, we have $D(1,T)=0$. By (2) of Theorem \ref{thm3.1},
\begin{align*}
PD_i(1,T,\Phi)\leq D(1,T)+\sup\varphi_1+C,
\end{align*}
which implies that $PD_i(1,T,\Phi)\leq\sup\varphi_1+C<+\infty$.
Therefore, $PD_i(T,\Phi)\geq 1$ by the
definitions of $PD_i(T,\Phi)$ $(i=1,2,3,4)$.
\end{proof}
Let $A$ be a positive constant  and $\Phi=(nA)_{n\in\mathbb N}$. It is obvious that
$\Phi=(nA)_{n\in\mathbb N}$ is an almost additive sequence.
\begin{thm}\label{thm4.2}
Let $(X,T)$ be a topological dynamical system and $s>0$. If $\Phi=(nA)_{n\in\mathbb N},A>0$,
then
\begin{eqnarray}\label{4.11}
PD_i(s,T,\Phi)=\left\{
\begin{array}{ll}
D(s,T),
&\mbox{\rm if} ~s>1,
\\
D(s,T)+A,
&\mbox{\rm if} ~s=1,
\\
\;
+\infty,&\mbox{\rm if}~0\leq s<1

~~~~~~~~~~~~~~~~~~~~~~~~~~~
\end{array}
\right.
\end{eqnarray}
for $i=1,2,3,4$.
\end{thm}

\begin{proof}
We first prove that the cases $i=1,4$. Let $\alpha\in\mathcal{C}_X$ and
$\beta$ is a subcover of $\bigvee\limits_{i=0}^{n-1}T^{-i}(\alpha)$. Let
$B\in\beta$. Note that
$\inf\limits_{x\in B}e^{\varphi_n(x)}=\inf\limits_{x\in B}e^{nA}$, which
implies that
\begin{align*}
q_n(T,\Phi,\alpha)&=\inf\left\{\sum\limits_{B\in\beta}\inf\limits_{x\in B}e^{\varphi_n(x)}|
\beta ~{\rm is~ a~ finite~ subcover~ of}~\bigvee\limits_{i=0}^{n-1}T^{-i}\alpha\right\}\\
&=e^{nA}\mathcal{N}(\bigvee\limits_{i=0}^{n-1}T^{-i}\alpha)
\end{align*}
and
\begin{align*}
p_n(T,\Phi,\alpha)&=\inf\left\{\sum\limits_{B\in\beta}\sup\limits_{x\in B}e^{\varphi_n(x)}|
\beta ~{\rm is~ a~ finite~ subcover~ of}~\bigvee\limits_{i=0}^{n-1}T^{-i}\alpha\right\}\\
&=e^{nA}\mathcal{N}(\bigvee\limits_{i=0}^{n-1}T^{-i}\alpha).
\end{align*}
Therefore,
\begin{align*}
PD_1(s,T,\Phi)=PD_4(s,T,\Phi)=\sup\limits_{\alpha\in\mathcal C_X}\limsup\limits_{n\to\infty}
\frac{1}{n^s}(\log\mathcal{N}(\bigvee\limits_{i=0}^{n-1}T^{-i}\alpha)+nA).
\end{align*}
Furthermore, we have
\begin{itemize}
   \item if $s>1$, $PD_1(s,T,\Phi)=PD_4(s,T,\Phi)=D(s,T)$,
   \item if $s=1$, $PD_1(s,T,\Phi)=PD_4(s,T,\Phi)=D(s,T)+A$,
   \item if $0\leq s<1$, $PD_1(s,T,\Phi)=PD_4(s,T,\Phi)=+\infty$.
\end{itemize}
This shows that (\ref{4.11}) holds for $i=1,4$. It follows from the inequality (\ref{2.2}) that  (\ref{4.11}) holds for $i=2,3$.
\end{proof}

\begin{thm}\label{thm4.3}
Let $(X,T)$ be a topological dynamical system and $\Phi=(nA)_{n\in\mathbb N},A\geq 0$.
Then
\begin{description}
\item[(1)]
If $A=0$, then $PD_i(T,\mathbf{0})=D(T)$, $i=1,2,3,4$.
\item[(2)]
If $A>0$ and $0\leq D(T)\leq 1$, then $PD_i(T,\Phi)=1$, $i=1,2,3,4$.
\item[(3)]
If $A>0$ and $D(T)>1$, then $PD_i(T,\Phi)=D(T)$, $i=1,2,3,4$.
\end{description}
\end{thm}

\begin{proof}
(1) If $A=0$, then $\Phi=\mathbf{0}$, which implies that $PD_i(s,T,\mathbf{0})=D(s,T)$
by Theorem \ref{thm3.1}. Therefore,
\begin{align*}
PD_i(T,\mathbf{0})=D(T),~(i=1,2,3,4).
\end{align*}

(2) Assume $0\leq D(T)\leq 1$. If $0\leq s<1$, then $PD_i(s,T,\Phi)=+\infty$ by (\ref{4.11}), which
implies that $PD_i(T,\Phi)\geq s$. Thus, $PD_i(T,\Phi)\geq 1$. Moreover, if $s>1$, then
$D(s,T)=0$. Furthermore, we have $PD_i(s,T,\Phi)=0$ by (\ref{4.11}), which implies that
$PD_i(T,\Phi)\leq 1$. Therefore,
\begin{align*}
PD_i(T,\Phi)=1, ~i=1,2,3,4.
\end{align*}

(3) Assume $D(T)>1$. If $0<s<D(T)$, then $D(s,T)=\infty$. Hence,
we have $PD_i(s,T,\Phi)=\infty$ by (\ref{4.11}), which implies that $PD_i(T,\Phi)\geq s$.
Again, if $1\leq D(T)\leq s,$ then $D(s,T)=0$. Furthermore, by (\ref{4.11}), we have
$PD_i(s,T,\Phi)=0$, which implies that $PD_i(T,\Phi)\leq s$. Therefore,
\begin{align*}
PD_i(T,\Phi)=D(T), ~i=1,2,3,4.
\end{align*}

\end{proof}

\begin{exmp}
Let $(\Sigma_2,\sigma)$ be a one-sided symbolic
dynamical system, where $\Sigma_2=\{x=(x_n)_{n=0}^{\infty}: x_n\in
\{0,1\}~\mbox{for every}~ n\}$,
$\sigma(x_0,x_1,x_2,\cdots)=(x_1,x_2,\cdots)$. Let $\Phi=(nA)_{n\in\mathbb N}$ and $A>0$.
Then $PD_{i}(T,\Phi)=1$, $i=1,2,3,4$.
\end{exmp}

Considering $\{0,1\}$ as a discrete space and putting product
topology on $\Sigma_2$, an admissible metric $\rho$ on
the space $\Sigma_2$ is defined by
\begin{eqnarray*}
\rho(x,y)=\sum\limits_{n=0}^{\infty}\frac{d(x_n,y_n)}{2^n},
\end{eqnarray*}
where \begin{eqnarray*}
d(x_n,y_n)=\left\{
\begin{array}{ll}
 0,
&\mbox{if} ~ x_n=y_n,
\\
\;
 \\
 1,&\mbox{if}~x_n\neq y_n,
~~~~~~~~~~~~~~~~~~~~~~~~~~~
\end{array}
\right.
\end{eqnarray*}
for $x=(x_0,x_1,\cdots),~y=(y_0,y_1,\cdots)\in\Sigma_2$.
By Robinson {\rm\cite{Robinson-book}},
$\Sigma_2$ is a compact metric space. From \cite{Cheng-Li},
$D(T)=1$. Hence, by (2) of Theorem \ref{thm4.3}, we have
$PD_i(T,\Phi)=1$, $i=1,2,3,4$.

\begin{exmp}
Let $(X,T)$ be a topological dynamical system and $\Phi=(nA)_{n\in\mathbb N},A>0$.
Then
\begin{description}
\item[(1)]
If $T:X\to X$ is a contractive continuous map, then $PD_i(T,\Phi)=1$, $i=1,2,3,4$.
\item[(2)]
If $X$ is an unit circle and $T:X\to X$ is a homeomorphism, then $PD_i(T,\Phi)=1$, $i=1,2,3,4$.
\end{description}
\end{exmp}
\noindent{\bf Proof of Example 4.2.} (1) Since $T:X\to X$ is a contractive continuous map,
it follows that $D(T)=0$ from \cite{Carvalho}. By (2) of Theorem \ref{thm4.3}, we have
$PD_i(T,\Phi)=1$, $i=1,2,3,4$.

(2) From \cite{Cheng-Li}, we have $D(T)=0$. Hence, by Theorem \ref{thm4.3} (2),
$PD_i(T,\Phi)=1$, $i=1,2,3,4$.

\noindent {\bf Acknowledgements.}    The work was supported by the Fundamental Research
Funds for the Central Universities (grant No. WK0010000035).
%%%%%%%%%%%%%%%%%%%%%%%%%%%%%%%%%%%%%%%%%%%%%%%%

\end{document}